\documentclass[12pt]{article}%
   
\usepackage{amsmath,enumerate}
\usepackage{amsfonts}
\usepackage{amssymb}

\newtheorem{theorem}{Theorem}[section]

\newtheorem{problem}[theorem]{Problem}
\newtheorem{proposition}[theorem]{Proposition}

\newenvironment{proof}[1][Proof]{\noindent\textbf{#1.} }
{\hfill \ \rule{0.5em}{0.5em}}

\newcommand{\ff}[1]{{\mathbb F}_{#1}}
\newcommand{\ffs}[1]{{\mathbb F}_{#1}^\star}
\newcommand{\ffx}[1]{\ff{#1}[X]}

\DeclareMathOperator{\Tr}{Tr}
\DeclareMathOperator{\Nm}{N}

\begin{document}

\title{Planar polynomials and an extremal problem of Fischer and Matou\u{s}ek}

\author{Robert S. Coulter\thanks{Department of Mathematical Sciences, University of Delaware, Newark, DE, 19716, USA.}
\and
Rex W. Matthews\thanks{6 Earl St., Sandy Bay, Tasmania 7005, Australia.}
\and
Craig Timmons\thanks{Department of Mathematics and Statistics, California State University Sacramento, USA. \texttt{craig.timmons@csus.edu}.
This author was was supported by the Simons Foundation (Grant \#359419)}
}

\maketitle

\abstract{
Let $G$ be a 3-partite graph with $k$ vertices in each part
and suppose that between any two parts, there is no 
cycle of length four.  Fischer and Matou\u{s}ek asked for the maximum number of triangles in such a graph.  
A simple construction involving arbitrary projective planes shows that there is such a graph with 
$(1 - o(1)) k^{3/2} $ triangles, and a double counting argument shows that one  
cannot have more than $(1+o(1))  k^{7/4} $ triangles.  
Using affine planes defined by specific planar polynomials over finite
fields, we improve the lower bound to $(1 - o(1)) k^{5/3}$.
}


\section{Introduction}

Let $n$ and $k$ be positive integers and write $[n]$ for $\{1,2, \dots , n \}$.    
If $\mathcal{F}$ is a family of functions from $[n]$ to $[2]$, then a 
set $A \subseteq [n]$ is called \emph{shattered} if given any function 
$g : A \rightarrow \{1, 2 \}$, there exists a function $f \in \mathcal{F}$ such that 
$f(a) = g(a)$ for all $a \in A$. 
The well-studied \emph{Vapnik-Chervonenkis dimension} 
or \emph{VC-dimension} of  $\mathcal{F}$ is the 
maximum size of a subset $A \subseteq X$ 
that is shattered by $\mathcal{F}$.  
A generalization of VC-dimension is 
the so-called Natarajan dimension.  
Let $\mathcal{F}$ be a collection of functions from $[n]$ to $[k]$.  
Given a set $A \subseteq [n]$, we say that $A$ is \emph{2-shattered} if 
for each $x \in A$, there is a pair $V_x \subseteq [k]$ such that 
for any choice of elements $c_x \in V_x$, there is an $f \in \mathcal{F}$ 
such that $f(x) = c_x$ for all $x \in A$.  
The family $\mathcal{F}$ has \emph{Natarajan dimension at most} $d$ if there is no 
subset $A \subseteq X$ with $d + 1$ elements that is 2-shattered by $\mathcal{F}$.

A natural question is given $n$ and $d$, how many functions can belong to $\mathcal{F}$ if the VC-dimension
of $\mathcal{F}$ is at most $d$?  Similarly, given $n$, $k$, and $d$, one can ask how large 
$\mathcal{F}$ can be if the Natarajan dimension of $\mathcal{F}$ is at most $d$.  
Fischer and Matou\u{s}ek \cite{fm} reformulated this problem as an interesting problem in extremal graph theory.
Given a collection of functions $\mathcal{F}$ from $[n]$ to $[k]$, we can view $\mathcal{F}$ as defining a
$n$-uniform $n$-partite hypergraph where each part has $k$ vertices.  The vertex set of 
this hypergraph is $[n] \times [k]$
and the edges are all sets of the form 
\[
\{ (1 , f(1) ) , ( 2 , f(2) ) , \dots , ( n , f(n) ) \},
\]
where $f \in \mathcal{F}$.  
A set $A \subseteq [n]$ is 2-shattered
if the subhypergraph of $\mathcal{F}$ induced by $A \times [k]$ 
contains  
a complete $|A|$-uniform, $|A|$-partite hypergraph with two vertices in each part.
For more on VC-dimension, Natarajan dimension, and its connection to hypergraphs, 
we refer the reader to \cite{fm} and the references therein.    

Fischer and Matou\u{s}ek showed that 
there is a family of functions from $[n]$ to $[3]$ with $3n$ elements and Natarajan dimension 1.  
Additionally, the $3n$ is best possible.  
This gives a solution to a 
special case of the above mentioned problem, but 
many cases remain open.  One of particular interest, mentioned explicitly in 
 \cite{fm}, is when $n = 3$, $d = 1$, and $k \geq 3$ is arbitrary.  The corresponding 
 extremal graph theory problem is as follows.  
 
 \begin{problem}\label{problem}
Let $G$ be a 3-partite graph with $k$ vertices in each part and 
suppose 
 that the bipartite graph between any two parts does not contain a 
 cycle of length four.  Determine how 
 many triangles can appear in such a graph.
 \end{problem}
 
While Problem \ref{problem} arose in the context of Natarajan dimension, 
given the recent activity on counting copies 
of a fixed graph $H$ in an $F$-free graph with $n$ vertices \cite{as, bg, grzesik, gl, hatami}, 
it is an interesting extremal problem in its own right.  
Let 
\[
\triangle (k)
\]
be the maximum number of triangles in a 3-partite graph with $k$ vertices in each 
part such that between any two parts, there is no cycle of length four.    
To our knowledge, the best known bounds on $ \triangle (k)$ are given in the next proposition.  

\begin{proposition}[Fischer, Matou\u{s}ek \cite{fm}]\label{fm prop}
The function $\triangle (k)$ satisfies 
\[
k^{3/2} - o ( k^{3/2} ) \leq \triangle (k) \leq k^{7/4} + O(k^{3/2} )
\]
as $k \rightarrow \infty$.  
\end{proposition}

For a proof of the upper bound, see \cite{fm}.  Our main result concerns the lower bound 
so we take a moment to sketch a proof.    
Assume that $q$ is a power of a prime.  Let $G(A,B)$ be the incidence graph of a 
projective plane of order $q$ and let $C$ be a set of $q^2 + q + 1$ vertices disjoint from $A \cup B$.  
Make a single vertex in $C$ adjacent to all vertices in $A$ and all vertices in $B$.  
This graph will be 3-partite with $q^2 + q + 1$ vertices in each part.  There will 
be no cycle of length four between any two parts.  The number of triangles in this graph is the number of edges between 
$A$ and $B$ which is $(q + 1)(q^2 + q + 1)$.  
Therefore,
\[
\triangle (q^2 + q + 1) \geq (q + 1)(q^2 + q + 1)
\]
whenever $q$ is a power of a prime.   
Our main result improves this lower bound.    

\begin{theorem}\label{main thm}
If $q$ is a power of an odd prime, then
\[
\triangle (q^6) \geq q^6 ( q^3 - 1)(q + 1).
\]
\end{theorem}

By Theorem \ref{main thm} and a standard density of primes argument, we have 
\[
\triangle (k) \geq (1 - o(1) )  k^{5/3} 
\]
as $k \rightarrow \infty$.  

To prove Theorem \ref{main thm}, we will use planar polynomials.
Planar functions were introduced by Dembowski and Ostrom \cite{do} in order to
construct affine planes with certain collineation groups.
Before defining planar polynomials, we introduce some notation.
We write $\ff{q}$ for the finite field with $q$ elements and $\ffs{q}$ for the
nonzero elements of $\ff{q}$.
The norm and trace maps from $\ff{q^3}$ to $\ff{q}$ will be denoted by
$\Nm$ and $\Tr$, respectively; that is, $x \in \ff{q^3}$, 
\begin{center}
$\Nm(x) = x^{1 + q + q^2}$~~ and~~ $\Tr(x) = x + x^q + x^{q^2}$.
\end{center}

Assume now that $q$ is a power of an odd prime.  
A polynomial $f \in \ffx{q}$ is a \emph{planar polynomial} 
if, for each $a \in \ffs{q}$, the map
\[
x \mapsto f(x  + a ) - f(x)
\]
is a bijection on $\ff{q}$.
Such polynomials can be used to construct affine planes  and consequently, they
can also be used to construct bipartite graphs without a cycle of length four.  
The simplest example of a planar polynomial is $f(X)  = X^2$, and this is
the smallest example of a class of planar monomials.
Let $\alpha, e$ be positive integers.
The monomial $f(X) = X^{q^{ \alpha} + 1}$ is planar over $\ff{q^e}$ if and
only if $\frac{e}{ \gcd (\alpha, e) }$ is odd, see \cite{cm}.
To obtain our lower bound, we consider planar monomials whose degree 
increases with $q$, specifically the monomial $X^{q+1}$ over $\ff{q^3}$.
The crucial algebraic ingredient used to prove Theorem \ref{main thm}
is derived from considering our graph construction using these monomials and
may be of independent interest. It reads as follows.

\begin{theorem}\label{polynomial theorem}
Let $q$ be a power of an odd prime.  For any $a \in \ffs{q^3}$, the polynomial 
\[
f_a (X) = X^{q + 1} + a^{-1} ( X^q + X ) + \textup N ( a^{-1} ) ( \textup{Tr}(a) - 2a )
\]
splits completely in $\ffs{q^3}$.  Furthermore, if $a \in \ffs{q}$, then 
$f_a (X)$ has a single root of multiplicity $q+1$, and if 
$a \in \mathbb{F}_{q^3} \backslash \mathbb{F}_q$, then the roots of $f_a (X)$ are all distinct.  
\end{theorem}

The graph proving the lower bound in Theorem \ref{main thm} 
is a 3-partite graph with $q^6$ vertices in each part, and the edge density between any two parts will be 
very close to $\frac{1}{q^3}$.
If we treated the edges as if they were placed randomly, we would 
expect roughly $q^{9}$ triangles, however, this graph contains at least $q^{10} -  O (q^9 )$ triangles.  
This is significantly more triangles than one might expect
and yet, the edges between the parts cannot be too unevenly distributed by 
the Expander Mixing Lemma.     

In the next section we prove Theorem \ref{polynomial theorem}.
The graph showing the lower bound of Theorem \ref{main thm} is defined
in Section 3, which also contains the proof of Theorem \ref{main thm}.


\section{Proof of Theorem 1.4}

The edges in the graph that we construct will be defined using the polynomial $X^{q+1} \in \ff{q^3}[X]$.  
Since this polynomial is planar, we will satisfy the condition of having no cycle of length four between two parts as the 
bipartite subgraph between any two parts is an affine plane.  
The difficult part is in counting the triangles.  This is where we require Theorem \ref{polynomial theorem} which we now prove.  

\bigskip

\begin{proof}[Proof of Theorem \ref{polynomial theorem}]
Let $a \in \ffs{q^3}$ and 
\[
f_a (X) = X^{q + 1} + a^{-1} ( X^q + X ) + \textup N ( a^{-1} ) ( \textup{Tr}(a) - 2a ).
\]
We first note that
\begin{align*}
f_a(-a^{-q}) &= a^{-q^2-q} - a^{-q^2-1} - a^{-q-1}
+ a^{-q^2-q-1} (a^{q^2}+a^q-a)\\
&= a^{-q^2-q} - a^{-q^2-1} - a^{-q-1}
+ a^{-q-1} + a^{-q^2-1} - a^{-q^2-q}\\
&=0,
\end{align*}
so that $-a^{-q}$ is a root of $f_a(X)$.
We now normalise $f_a(X)$ with respect to this root.
We have
\begin{align*}
f_a(X-a^{-q})
&= (X^q-a^{-q^2})(X-a^{-q}) + a^{-1} (X^q + X)\\
&\quad - a^{-1} (a^{-q^2} + a^{-q}) + \Nm(a^{-1})(\Tr(a)-2a)\\
&= X^{q+1} + X^q(a^{-1} - a^{-q}) + X (a^{-1}-a^{-q^2}) +f_a(-a^{-q})\\
&= X^{q+1} + X^q(a^{-1} - a^{-q}) + X (a^{-1}-a^{-q^2})\\
&= X \left(X^q + X^{q-1} (a^{-1} - a^{-q}) + (a^{-1}-a^{-q^2})\right).
\end{align*}
Set $h(X)=X^q + X^{q-1} (a^{-1} - a^{-q}) + (a^{-1}-a^{-q^2})$, so that
$f_a(X-a^{-q}) = X\, h(X)$.

The root $-a^{-q}$ will be a multiple root of $f_a(X)$ if and only if 0 is a
root of $h(X)$. This occurs only when $a\in\ffs{q}$, in which case $h(X)=X^q$.
Consequently, $f_a(X)=(X+a^{-q})^{q+1}$, which establishes 
Theorem \ref{polynomial theorem} in the case that $a\in\ffs{q}$.

For the remainder of the proof, assume $a\in\ff{q^3}\setminus\ff{q}$.
We know from the above discussion that $-a^{-q}$ is not a multiple root of
$f_a(X)$. 
The reciprocal polynomial of $h(X)$ is 
\begin{equation*}
X^q h(X^{-1}) = (a^{-1} - a^{-q^2}) X^q + (a^{-1} - a^{-q}) X + 1.
\end{equation*}
Let $L(X)=(a^{-1} - a^{-q^2}) X^q + (a^{-1} - a^{-q}) X$.
The polynomial $L(X)$ is a linearized polynomial, and as it has a non-zero
$X$ term, it has no multiple roots. (For this and many other results on
linearized polynomials, see Lidl and Niederreiter \cite{blidl83}, Chapter 3.)
Indeed, it can be seen from the identity
\begin{align*}
L(X) &= (a^{-1} - a^{-q^2}) X^q + (a^{-1} - a^{-q}) X\\
&= (a^{-1} - a^{-q^2}) X^q - (a^{-1} - a^{-q^2})^q X,
\end{align*}
that $L(X)$ splits completely in $\ff{q^3}$, its roots being given by
$x=\alpha (a^{-1} - a^{-q^2})$, with $\alpha\in\ff{q}$.
Using the additive properties of linearized polynomials, it follows that if
$x_c\in\ff{q^3}$ satisfies $L(x_c)=c$ for some $c\in\ff{q^3}$, then
$L(x_c + \alpha (a^{-1} - a^{-q^2}))=c$ for any $\alpha\in\ff{q}$.
Thus, if $L(X)-c$ has a root in $\ff{q^3}$, then it splits completely over
$\ff{q^3}$ with distinct roots.
Given the relationship between $f_a(X)$, $h(X)$ and $L(X)$, we therefore have
$f_a$(X) splits completely, with distinct roots, over $\ff{q^3}$ if and only if
$L(X)+1$ has a root in $\ff{q^3}$.
We will show something stronger; we shall prove that for any 
$\alpha \in \ff{q}$, $L(X)-\alpha$ splits
completely, with distinct roots, in $\ff{q^3}$.

Fix $\alpha\in\ff{q}$ and suppose $L(x)=\alpha$ holds for some $x\in\ff{q^3}$.
Then $L(x)^q = \alpha$ also.
Hence,
\begin{align*}
0 &= L(x)^q - L(x)\\
&= 
(a^{-q} - a^{-1}) x^{q^2} + (a^{-q} - a^{-q^2}) x^q -
(a^{-1} - a^{-q^2}) x^q - (a^{-1} - a^{-q}) x\\
&= (a^{-q} - a^{-1}) (x^{q^2}+x^q+x)\\
&= (a^{-q} - a^{-1}) \Tr(x),
\end{align*}
and so $\Tr(x)=0$.
This argument can be reversed, proving $L(x)\in\ff{q}$ if and only if
$\Tr(x)=0$.
As there are $q^2$ elements $x\in\ff{q^3}$ for which $\Tr(x)=0$, we know that,
counting multiplicities, $L(x)\in\ff{q}$ for $q^2$ choices of $x$.
However, the degree of $L(X)$ is $q$, and so the polynomial $L(X)-\alpha$ can
have at most $q$ roots for any fixed $\alpha\in\ff{q}$.
Since we have exactly $q$ choices for $\alpha\in\ff{q}$, the polynomial
$L(X)-\alpha$ must have exactly $q$ distinct roots for each $\alpha\in\ff{q}$.
In particular, $L(X)+1$ does, proving that $f_a(X)$ has $q+1$ distinct roots whenever
$a\in\ff{q^3}\setminus\ff{q}$.
\end{proof}


\section{Proof of Theorem 1.3}

We begin this section by defining the graph that implies the lower bound asserted by Theorem \ref{main thm}.    

\bigskip

\textbf{The Construction}: Let $q$ be a power of an odd prime.
Choose $a \in \mathbb{F}_{q^3} \backslash \mathbb{F}_q$ 
so that 
\begin{equation}\label{constraint}
a \textup{N}(a^{-1} ) ( \textup{Tr}(a) - 2a ) - 1 \neq 0
\end{equation}
and $-1$ is not a root of $f_a (X)$.  The equation 
$a \textup{N}(a^{-1} ) ( \textup{Tr} (a) - 2a ) - 1 = 0$ is equivalent to 
\[
( a^{-1} )^{q^2 + q - 1} +(a^{-1} )^{q^2} +  (a^{-1})^q - 1  = 0
\]
so there are at most $q^2+q -1 $ elements of $\mathbb{F}_{q^3} \backslash \mathbb{F}_q$ for which 
(\ref{constraint}) fails.
Similarly, $-1$ is a root of $f_a (X)$ if and only if 
$a^{-1}$ is a root of $X^{q^2 + q} - X^{q^2 + 1} - X^{q + 1} + 2X - 1$. 
Since $q^3 - q - ( q^2 + q)  - ( q^2 + q - 1) \geq 1$ for all $q \geq 3$, such an $a$ exists.  
We also remark that since $a \in \mathbb{F}_{q^3} \backslash \mathbb{F}_q$, 
0 is not a root of $f_a (X)$ as $0 = f_a (0)$ implies that $\textup{Tr}(a)  = 2a$ which, in turn, implies 
$a \in \mathbb{F}_q$.    

Let 
\begin{center}
$f(X) = ( a - 1) X^{q+1} $, $g(X) = ( a \textup{N} (a^{-1} )  (\textup{Tr}(a) - 2a ) - 1) X^{q+1}$, and 
$h(X) = X^{q+1}$.
\end{center}
Each of the polynomials $f(X)$, $g(X)$, and $h(X)$ are nonzero planar polynomials over $\mathbb{F}_{q^3}$.  
Let $A$, $B$, and $C$ be disjoint copies of $\mathbb{F}_{q^3} \times \mathbb{F}_{q^3}$.
Elements in $A$ are denoted by $(x,y)_A$ and the same goes for elements in $B$ and $C$.    
Let $G_{q}(a)$ be the graph whose vertex set is $A \cup B \cup C$, where for all $x,y \in \mathbb{F}_{q^3}$ and 
$z \in \ffs{q^3}$, 
\begin{itemize}
\item $(x,y)_A$ is adjacent to $(x + z , y + f(z) )_B$, 
\item $(x,y)_B$ is adjacent to $(x + z , y + g(z) )_C$, and 
\item $(x,y)_C$ is adjacent to $(x + z , y + h(z) )_A$.
\end{itemize} 

\bigskip

Using Theorem \ref{polynomial theorem}, we now prove the following which implies Theorem \ref{main thm}.  

\begin{theorem}\label{graph thm}
The graph $G_{q}(a)$ is a 3-partite graph with $q^6$ vertices in each part and there is no cycle of length four 
between two parts.  Furthermore, 
the number of triangles in $G_{q} (a)$ is at least $q^6 ( q^3 - 1)(q + 1)$.  
\end{theorem}
\begin{proof}
It is clear that $G_{q}(a)$ is 3-partite with $q^6$ vertices in each part.  Since each of $f$, $g$, and $h$ are planar polynomials, 
there is no cycle of length four between any two parts.  This is easily deduced from Lemma 12 of \cite{do}. 
It remains to show that $G_{q}(a)$ has at least $q^6 ( q^3 - 1)(q + 1)$ triangles.  

Let $\xi_1 , \dots , \xi_{q+1}$ be distinct roots in $\ff{q^3}$ of 
\[
X^{q+1} + a^{-1} ( X^q + X ) + \textup{N}(a^{-1}) ( \textup{Tr}(a) - 2a ).
\]
These roots exist by Theorem \ref{polynomial theorem}.  Choose a root $\xi_j$ and 
let $z_2$ be any element of $\ffs{q^3}$.  Define $z_1$ by $z_1 = \xi_j z_2$.  
We then have
\[
a ( z_1 z_2^{-1} )^{q + 1} +  ( z_1 z_2^{-1} )^q + ( z_1 z_2^{-1} ) + a \textup{N}(a^{-1}) ( \textup{Tr}(a) - 2a )=0
\]
which is equivalent to 
\begin{equation*}
(a - 1) (z_1 z_2^{-1} )^{q + 1} + (z_1 z_2^{-1} + 1)^{q+1} + a \textup{N}(a^{-1} ) ( \textup{Tr}(a) - 2a ) - 1 = 0.
\end{equation*}
Since $q+1$ is even and $z_2 \neq 0$, we can rewrite this equation as 
\begin{equation}\label{graph thm eq}
(a - 1)z_1^{q+1} + ( a \textup{N}(a^{-1} ) ( \textup{Tr}(a) - 2a ) - 1) z_2^{q+1} + ( - z_1 - z_2 )^{q+1} = 0.
\end{equation}
If we let $z_3 = -z_1 - z_2$, then from the definition of $f$, $g$, and $h$, we have from (\ref{graph thm eq}) that
\[
f(z_1) + g(z_2) + h(z_3) = 0.
\]
Observe that $z_1$, $z_2$, and $z_3$ are all non-zero since $\xi_j \notin \{ 0 , -1 \} $.  
Thus, for any $(x,y) \in \mathbb{F}_{q^3} \times \mathbb{F}_{q^3}$, the vertices 
\[
(x,y)_A , (x + z_1  , y + f(z_1) )_B , (x + z_1 + z_2 , y + f(z_1) + g(z_2) )_C
\]
form a triangle since 
\[
(x,y)_A = ( x + z_1 + z_2 + z_3 , y + f(z_1) + g(z_2) + h(z_3  ) )_A.
\]
There are $q+1$ choices for $\xi_j$, $q^3 - 1$ choices for $z_2$ (which then determines $z_1$ and $z_3$), 
and $q^6$ choices for $(x,y)$.  Altogether, this gives $q^6 ( q^3  - 1)(q + 1)$ triangles in $G_{q}(a)$
completing the proof of Theorem \ref{graph thm}.
\end{proof}

We make some final remarks.
Constructions using planar monomials and similar to the one used to prove
Theorem \ref{main thm} have appeared elsewhere.  
Allen, Keevash, Sudakov, and Verstra\"{e}te \cite{aksv} use the planar monomial
$X^2$ over $\mathbb{F}_q$ to construct $\{K_3 , K_{2,3} \}$-free graphs with
many edges.    
Other instances include \cite{tv} and \cite{t}, but like \cite{aksv}, these
papers all use $X^2$.
Using the planar monomial $X^2$ in place of $X^{q+1}$ in our construction only
leads to an improvement upon the lower bound of Proposition \ref{fm prop} by a
constant factor of 2.
One of the novelties of our approach is the use of a planar polynomial that is
more complicated than $X^2$.
We are not aware of another instance in extremal graph theory where an existing
result was improved upon by considering planar polynomials other than $X^2$.  
There is one further class of planar monomials known -- the monomial
$X^{(3^\alpha+1)/2}$ is planar over $\ff{3^e}$ if and only if
$\gcd(\alpha,2e)=1$, see \cite{cm}.
Computational evidence suggests replacing $X^{q+1}$ with these polynomials will
not provide an improvement to Theorem \ref{main thm}.


\section{Acknowledgment}

The third listed author would like to thank Jacques Verstra\"{e}te and Jason Williford for helpful discussions.    


\end{document}